\documentclass[11pt]{amsart}
\usepackage[toc,page]{appendix}
\usepackage{amsmath,amssymb,latexsym}
\usepackage[small]{caption}
\usepackage{graphicx,color,mathrsfs,tikz}
\usepackage{subfigure,color}
\usepackage{cite}
\usepackage[colorlinks=true,urlcolor=blue,
citecolor=red,linkcolor=blue,linktocpage,pdfpagelabels,
bookmarksnumbered,bookmarksopen]{hyperref}
\usepackage[italian,english]{babel}
\usepackage{enumitem}
\usepackage[left=2.7cm,right=2.7cm,top=2.9cm,bottom=2.9cm]{geometry}
\usepackage[hyperpageref]{backref}

\usepackage[colorinlistoftodos,prependcaption]{todonotes}

\numberwithin{equation}{section}
\newtheorem{theorem}{Theorem}[section]
\newtheorem{proposition}[theorem]{Proposition}
\newtheorem{lemma}[theorem]{Lemma}

\newtheorem{corollary}[theorem]{Corollary}
\newtheorem{definition}[theorem]{Definition}
\theoremstyle{definition}

\newtheorem{remark}[theorem]{Remark}
\renewcommand{\epsilon}{\eps}

\newcommand{\HC}{\mathcal{HC}}

\newcommand{\C}{{\mathcal C}}

\newcommand{\R}{{\mathbb R}}

\newcommand{\dvg}{{\rm div}}

\newcommand{\eps}{\varepsilon}

\newcommand{\Tr}{{\rm Tr}}

\newcommand{\pnorm}[2][]{\if #1'' \left|#2\right|_p \else \left|#2\right|_{#1} \fi}

\renewcommand{\theta}{\vartheta}

\newcommand{\Rn}{{\mathbb R^{n}}}
\newcommand{\eqlab}[1]{\begin{equation}  \begin{aligned}#1 \end{aligned}\end{equation}} 
\newcommand{\bgs}[1]{\begin{equation*} \begin{aligned}#1\end{aligned}\end{equation*}} 
 \newcommand{\syslab}[2] []  {\begin{equation}#1  \left\{\begin{aligned}#2\end{aligned}\right.\end{equation}} 
  \newcommand{\sys}[2][]{\begin{equation*}#1  \left\{\begin{aligned}#2\end{aligned}\right.\end{equation*}}

\def\XXint#1#2#3{{\setbox0=\hbox{$#1{#2#3}{\int}$ }
\vcenter{\hbox{$#2#3$ }}\kern-.6\wd0}}

\title[Approximate convexity principles and applications]{Approximate convexity principles and \\
	applications to PDEs in convex domains}

\author{Claudia Bucur}
\author{Marco Squassina}

\address[C. Bucur]{Dipartimento di Matematica e Fisica \newline\indent
	Universit\`a Cattolica del Sacro Cuore \newline\indent
	Via dei Musei 41, I-25121 Brescia, Italy}
\email{claudia.bucur@aol.com}

\address[M.\ Squassina]{Dipartimento di Matematica e Fisica \newline\indent
	Universit\`a Cattolica del Sacro Cuore \newline\indent
	Via dei Musei 41, I-25121 Brescia, Italy}
\email{marco.squassina@unicatt.it}

\thanks{The authors are members
	of {\em Gruppo Nazionale per l'Analisi Ma\-te\-ma\-ti\-ca, la Probabilit\`a e le loro Applicazioni} (GNAMPA) 
	of the {\em Istituto Nazionale di Alta Matematica} (INdAM)}



\begin{document}

\begin{abstract}
	We obtain  approximate convexity principles for solutions to some classes of nonlinear elliptic partial differential equations in convex domains involving 
	approximately concave nonlinearities. Furthermore, we provide some applications to some meaningful special cases.
\end{abstract}
\maketitle

\section{Introduction}

Convexity properties of solutions to elliptic partial differential equations in convex domains are a fascinating subject. One of the first results in this direction goes back to the work of Brascamp and Lieb \cite{BL}  from 1976, where
they proved that the logarithm function applied to the first eigenfunction of the Laplace operator with zero Dirichlet
boundary conditions in a convex domain is concave. Notice that the first eigenfunction itself \emph{is not} 
concave in any domain (as it can be easily seen), thus considering a transformation (in this case, taking the logarithm) of the solution is necessary. Previously, in 1971, Makar-Limanov \cite{Makar} had proved 
that if $u$ is the positive solution to the torsion equation $\Delta u+1=0$ in the convex domain $\Omega$, then $\sqrt{u}$ is concave.
Later, at the beginning of the eighties, Korevaar \cite{kore,kore2} and Kennington \cite{kenn} were able to derive these results from general convexity principles (see also \cite{kaw,CS,CF}). 
Given a convex domain $\Omega\subset \Rn$ and a function $u:\bar\Omega\to\R$, 
these convexity principles are essentially maximum principles for the auxiliary function
\eqlab{\label{cu}
\mathcal C_u(y_1,y_3,\lambda):= 
u(\lambda y_1			+ (1-\lambda)y_3)
-\lambda u(y_1)-(1-\lambda)u(y_3),
}
for $y_1,y_3\in\bar\Omega$ and $\lambda\in [0,1]$. 
Positivity (negativity) of $\mathcal C_u$ in $\bar\Omega\times\bar\Omega\times [0,1]$ is equivalent to concavity (convexity) of the function $u$. 

As a by product of the general theory, some results about
concavity of positive solutions of notable semilinear problems can be obtained.
For instance (see \cite[Theorem 4.2]{kenn}), if $n\geq 2$, $\gamma\in (0,1)$, $\Omega$ is a bounded convex domain of $\R^n$ that satisfies an interior ball condition and $u\in C^2(\Omega)\cap C(\bar\Omega)$ is a solution to
\eqlab{ \label{jjj}
\Delta u+u^\gamma=0,  \qquad \text{$u=0$ \,\, on $\partial\Omega$},
\qquad \text{$u>0$ \,\,  in $\Omega$},
}
then $u^{(1-\gamma)/2}$ is concave in $\bar\Omega$. Also (see \cite[Theorem 4.1]{kenn}), if $\gamma\geq 1$ and $u\in C^2(\Omega)\cap C(\bar\Omega)$ is a solution to
\eqlab{\label{jjj1}
\Delta u+f(x)=0, 
 \qquad \text{$u=0$ \,\, on $\partial\Omega$},
\qquad \text{$u>0$ \,\,  in $\Omega$},
}
for some nonnegative $f\colon \Omega \to \R$ such that $f^\gamma$ is concave,
then 
$$
\text{$u^{\alpha}$ is concave in $\bar\Omega$ if 
	$0<\alpha\leq\frac{\gamma}{1+2\gamma}$,}
$$
and the upper bound is sharp (cf.\ Property 2 and Theorem 6.2 of \cite{kenn}). Roughly speaking, some form of concavity on the nonlinear term forces a suitable power of the positive solution 
to be concave. Similar statements hold in some 
cases when one takes the logarithm of the first eigenvalue of the Laplace or $p$-Laplace operator  
with Dirichlet boundary conditions, see \cite{saka}. See also \cite{CS1,CS2} for general
concavity principles for some classes of fully nonlinear elliptic problems, obtained 
with different techniques compared to \cite{kore,kenn}.

It is rather natural to wonder what happens if the concavity of the nonlinear term 
is broken
down by a small perturbation. Is then the corresponding solution of
the problem convex up to a small perturbation function of proportional size? 

The answer is affirmative and it follows from approximate convexity principles 
that we prove in Theorems \ref{first} and \ref{thm3}, in combination with constraints
furnished by the boundary conditions of the problems under consideration. 
As a consequence of the approximate convexity principles
we  obtain the corresponding results of approximate convexity of perturbed problems like the ones in \eqref{jjj} and \eqref{jjj1}.
	
	The main applications of this paper are given in the following informal terms. 
	
	Let $n\geq 2$, $\gamma\in [0,1]$, $\Omega$ a bounded strictly convex domain of $\R^n$ 
	that satisfies an interior ball condition.
	 Let $u\in C^2(\Omega)\cap C(\bar\Omega)$ be a solution to
	\[ \Delta u+u^\gamma
	-	 u^{\frac{1+\gamma}2}g(u)=0,  \qquad \text{$u=0$ on $\partial\Omega$},
	\qquad \text{$u>0$ in $\Omega$.}
	\]
Then under some assumptions of $\delta$-approximate harmonic convexity 
and monotonicity of $g$, and requiring that the nonlinear term $u^\gamma
	-	 u^{(1+\gamma)/2}g(u)$  stays positive, 
		there exists a concave function $v$ and a positive constant $C$ such that 
	\begin{enumerate}
		\item if $\gamma\in [0,1)$, then
	\[
	\|u^{(1-\gamma)/2}-v\|_{L^\infty(\Omega)} \leq C\delta,
	\] 
	\item if $\gamma=1$, then
	\[
	\|\log u-v\|_{L^\infty(\Omega)}\leq C\delta.
	\] 
\end{enumerate}

	This main application is proved in Theorem \ref{powerC} and Corollary \ref{logconc} (some less restrictive hypothesis will be required  in the respective results). Furthermore, we provide a result for a problem like the one in \eqref{jjj1}, as follows.

	Let $n\geq 2$,  $\Omega$ be a bounded convex domain of $\R^n$.
	Let $u\in C^2( \Omega)\cap C(\bar \Omega)$ be a solution to
	$$
	\Delta u+   f(x) - u^{\frac{1+\gamma}{1+2\gamma} } g(x)=0,  \qquad \text{$u=0$ on $\partial\Omega$},
	\qquad \text{$u>0$ in $\Omega$}.
	$$ 
	Asking hypothesis on the concavity and strict positivity of $f$, approximate harmonic convexity of $g$, and requiring that the nonlinearity stays positive, there exists a concave function and $C$ such that  
	$v\colon \Omega \to \R$  
	\[ 
	\|u^{\frac{\gamma}{1+2\gamma}} -v\|_{L^\infty(\Omega)} \leq C \delta.
	\] 
	This result is given in Theorem \ref{fffh}.

In the rest of the paper, we introduce the framework and state the approximate convexity principles in Section \ref{GeneralR}. Boundary conditions of particular problems (that we use in Section \ref{bdry}) will allow us to give some explicit examples in the last Section \ref{examples}.

\section{Approximate Convexity Principles}
\label{GeneralR}

Let $\Omega\subset \Rn$ be a convex domain (i.e. a connected open set) here and in the rest of the paper. {We denote by $y_1,y_3$ the generic points of $\Omega$ and by $y_2$ their convex combination, precisely}
\bgs{ & \mbox{ for any } y_1,y_3\in \bar \Omega, \, \lambda \in [0,1] \qquad y_2= \lambda y_1+(1-\lambda)y_3.} 
We adopt the same notation for $s_1,s_3\in \R$, denoting $s_2\in \R$ as their convex combination.

Let $u:\bar\Omega\to\R$ and $\delta>0$. As in \cite{HU}, we say that $u$ is 
 $\delta$-convex in $\bar \Omega$ if
$$\mathcal C_u(y_1,y_3,\lambda)\leq \delta,\quad \,\,
\text{for all $y_1,y_3\in\bar\Omega$ and any $\lambda\in [0,1]$,}
$$  
where $\C_u $ is defined in \eqref{cu}.
We say that $u$ is $\delta$-concave  if
$-u$ is $\delta$-convex (notice also that $\C_{-u}=-\C_u$).
Also, with an abuse of notation,  for $y_1,y_3\in\bar\Omega$, $s_1,s_3 \in \R$ and  $\lambda\in [0,1]$ we define 
\eqlab{\label{cg} 
		&  \C_g ( (y_1,s_1),(y_3,s_3),\lambda) = g(y_2,s_2) - \lambda g(y_1,s_1) -(1-\lambda) g(y_3,s_3)
	 }
		as the convexity function of some $g\colon \bar \Omega\times \R\to \R$, jointly in its two variables. 
		We write also
		\bgs{
		\C_{g(\cdot, u(\cdot))}(y_1,y_3,\lambda)= \C_g ( (y_1,u(y_1)),(y_3,u(y_3)),\lambda) 
	}
				as the convexity function of $g\colon \bar \Omega \times \R \to \R$ jointly in two variables, along $u\colon \bar \Omega\to \R$.
				\begin{remark}We make a remark on the notation adopted in the course of this paper. If $g$ depends only on the variable $y\in \bar \Omega$ then the two notions in \eqref{cu} and \eqref{cg} coincide. Nonetheless, we still use the notation $\C_g$ when the function $g$ depends only on $s\in \R$, and in general to denote the convexity function in one, or jointly in two variables. We point out once more that the notation  $\C_{g(\cdot, u(\cdot))}$ is referred to the joint convexity of $g$, and contains no information about the convexity of $u$ itself.
		\end{remark}
		\noindent We say that $g$ is jointly convex in $\bar \Omega$
		if and only if for all $ (y_1,s_1),(y_3,s_3) \in \bar \Omega\times \R  $  and any $\lambda \in [0,1]  $ we have that  $ \C_g ((y_1,s_1),(y_3,s_3),\lambda) \leq 0$, that $g$ is jointly $\delta$-convex if 
		\[ \C_g(  (y_1,s_1),(y_3,s_3),\lambda) \leq \delta\]
and is jointly $\delta$-concave if $-g$ is jointly $\delta$-convex.
		{ In particular, $g$ is jointly convex along $u$ if
			\[\C_{g(\cdot, u(\cdot))}  (y_1,y_3,\lambda) \leq \delta, \qquad \,\mbox{ for all } y_1,y_3 \in \bar \Omega \mbox{ and any }\lambda \in [0,1]  \]
		and  $g$ is jointly $\delta$-concave along $u$ when $-g$ is jointly $\delta$-convex along $u$. Of course, asking that $g$ is concave along $u$ is a refinement, and joint concavity of $g$ implies the concavity along $u$.}

		Furthermore, we define the {\em harmonic convexity function} jointly in the two variables $(y,s)$ of the function $g\colon \bar\Omega\times \R\to \R$ exclusively when 
		\eqlab{ \label{hcon} (1-\lambda )g(y_1,s_1) + \lambda g(y_3,s_3)>0 \qquad \mbox{ or } \qquad g(y_1,s_1) =  g(y_3,s_3)=0.
		}
	This may seem a little weird to the reader, but its definition is justified by the use we make in the rest of the paper, in particular in Lemma \ref{mainlem2} (we note also that this definition coincides with the one given in  \cite{kenn}). Thus, the definition is given for positive functions $g$, or changing sign functions that satisfy at the given point $((y_1,s_1),(y_3,s_3),\lambda)$ one of the conditions in \eqref{hcon}. Notice also that none of these conditions hold if $g<0$. 
			We thus define
		\sys[ \HC_{g} ( (y_1,s_1),(y_3,s_3),\lambda):=]{ 
		& g(
		y_2,s_2)
		- \frac{g(y_1,s_1) g(y_3,s_3)}{(1-\lambda )g(y_1,s_1) + \lambda g(y_3,s_3)},
		\\
		& \qquad\qquad\qquad\qquad \mbox{ if }   (1-\lambda )g(y_1,s_1) + \lambda g(y_3,s_3)>0
		\\
		 & g(
		 y_2,s_2),\qquad\quad\;\;\, \,\mbox{ if } 
		  g(y_1,s_1) =  g(y_3,s_3)=0.
		  }
		In general, we say that $g$ is $\delta$-{\em harmonic concave} ($\delta$-{\em harmonic convex}) if $\mbox{ for all } (y_1,s_1),(y_3,s_3) \in \bar\Omega\times \R \mbox{ and any }\lambda \in [0,1]$ satisfying one of the two conditions in \eqref{hcon} we have that
		\[
		\HC_g  ( (y_1,s_1),(y_3,s_3),\lambda) \geq -\delta,  \qquad\qquad  (\HC_g  ( (y_1,s_1),(y_3,s_3),\lambda) \leq \delta ).
		\]		
It is readily seen that 
		a {\em positive} concave function $g$ is harmonic concave.
		We notice also that the simple inequality
		\[ \frac{g(y_1,s_1) g(y_3,s_3)}{(1-\lambda )g(y_1,s_1) + \lambda g(y_3,s_3)} \leq \lambda  g(y_1,s_1) + (1-\lambda) g(y_3,s_3)\]
		that holds whenever $(1-\lambda )g(y_1,s_1) + \lambda g(y_3,s_3)>0$ 
		implies that
				\[
			\HC_{g} ( (y_1,s_1),(y_3,s_3),\lambda)\geq  \C_{g} ( (y_1,s_1),(y_3,s_3),\lambda).
		\]
		In particular, all positive $\delta$-harmonic convex functions are $\delta$-convex, and all positive $\delta$-concave are also $\delta$-harmonic concave.
{	
We denote the harmonic convexity function of $g$ along $u$ as
\[ \HC_{g(\cdot, u(\cdot))} (y_1,y_3,\lambda)  = \HC_g  ( (y_1,u(y_1)),(y_3,u(y_3)),\lambda) .\]
		We say that $g$ is $\delta$-{\em harmonic concave} ($\delta$-{\em harmonic convex}) along $u$ if $\mbox{ for all } y_1,y_3 \in \bar \Omega \mbox{ and any }\lambda \in [0,1]$ satisfying one of the two condition in \eqref{hcon}, it holds that
		\[
		\HC_{g(\cdot, u(\cdot))} (y_1,y_3,\lambda)  \geq -\delta, \qquad \, \qquad  (\HC_{g(\cdot, u(\cdot))} (y_1,y_3,\lambda)   \leq \delta).
		\] 
		}

	As expected from the previous work in the literature (for instance \cite{kore,kenn,kaw}), the convexity of solutions of a second order elliptic problems with a nonlinear term in a convex domain depends solely on the convexity and the monotonicity of the nonlinearity. We give in the next lemma a quantitative estimate of the convexity function of the solution. \\
Let us also mention that, as in \cite{kaw,kore,kenn} it is crucial that the second order coefficients depend only on the gradient of the solution. To our knowledge, convexity principles that allow a dependence on the solution itself or $x$ are not available. 

For the sake of clarity, we give the next definition. 

\begin{definition}\label{x2}
We say that the triple  $(x_1,x_3,\lambda)$  is an interior point for $\C_u$ if  each of $x_1,x_2,x_3 $ is in $ \Omega$ with $x_2=\lambda x_1+(1-\lambda)x_3$, while we say that the point is on the boundary if at least one $x_1,x_2,x_3 $ belongs to $\partial  \Omega $. 
\end{definition} 

Here and in the rest of the section we consider $a_{ij}\colon \Rn \to \R$ measurable functions for all $i,j=1, \dots,n$ and $b\colon \Omega\times\R\times\Rn$ derivable in the second variable, on its domain of definition. Moreover,  we write $[A,B]$ to denote the non-orientated segment from $A$ to $B$. 

\begin{lemma}\label{mainlem} 
We consider the equation in $\Omega$ 
	\eqlab{ \label{eqnu} \mathcal L u=0\quad \text{in $\Omega$},\qquad  \mathcal L u:= \sum_{i,j=1}^n a^{ij}(Du) \partial^2_{ij} u -b(x,u,D u).}
	Let $u\in C^2(\Omega)\cap C(\bar\Omega)$ be a solution of \eqref{eqnu}
	.
We assume that
		\eqlab{ \label{symmatr}
						 A=[a^{ij}(\xi)]_{i,j} \mbox{ is symmetric, positive defined for all } \xi\in\R^n.}
				 	Then, if $\C_u$ achieves a positive interior maximum at $(x_1,x_3,\lambda )\in \Omega\times \Omega \times [0,1]$ 
				 	and there exists $\beta >0$ such that
	\eqlab{\label{mon}
	\inf_{\xi \in [u(
	x_2), \lambda u(x_1) +(1-\lambda)u(x_3)] } \partial_u b(
	x_2,\xi ,Du(x_1)) \geq  \beta 
	}
then 
\[\C_u(x_1,x_3,\lambda ) \leq \frac1{\beta}\, 
{\C_{-b(\cdot, u(\cdot), D u(x_1))}( x_1,x_3, \lambda) }
 .
\]
\end{lemma} 	
We follow in the next proof the main ideas from \cite[Lemma 1.4]{kore}.  We remark that this Lemma contributes to the result in Theorem \ref{first}, which affirms that the $\delta$-concavity of $b$ {along $u$} implies the $\delta$-convexity of the solution $u$. Requiring thus  a positive maximum of $C_u$ is natural (otherwise, $C_u\leq0$ gives that $u$ is convex, and there would be nothing else to prove). 
	\begin{proof}
	We consider $x_1\neq x_3$ and $\lambda \in (0,1)$ (i.e., $x_2$ does not coincide with $x_1$ or $x_3$), otherwise $\mathcal C_u=0$, which gives that  $u$ is convex. Given that $(x_1,x_3,\lambda)$ is a interior maximum  point, we have 
	\[ D_{y_1} \mathcal C_u (x_1,x_3,\lambda)= D_{y_3} \mathcal C_u (x_1,x_3,\lambda)=0,\]
	therefore we  may denote 
		\[ D u(x_1)=D u(x_2)=D u(x_3):=z.\]
	Take now for $v\in\R^n$
	\[ 
			\bar {\mathcal C}(v):= u(x_2+v) -\lambda u(x_1+v)-(1-\lambda) u(x_3+v).
				\]
	Since $v=0$ gives a maximum, we have that
	\[ 
		D_v \bar {\mathcal C}(0)=0 \qquad 
		\mbox{ and } \qquad [D^2_v \bar {\mathcal C}(0)] \leq 0.
	\]
	Here $[D^2_v \bar {\mathcal C}(0)]$ denotes the Hessian with respect to $v$ of $\bar {\mathcal C}$ at zero. Also notice that
	\[ 
		(D^2_v \bar {\mathcal C}(0))_{ij}= \partial_{ij}^2 u(x_2)
		-\lambda  \partial_{ij}^2 u(x_1) -(1-\lambda) \partial_{ij}^2u(x_3).
		\]
	Since $A$ is symmetrical and positive defined, we get that
	\[
			A [D^2_v \bar {\mathcal C}(0)]  \leq 0,
			\]
			hence 
			\eqlab {\label{kkk1}
			\sum_{i,j=1}^n a^{ij} (z)\left(\partial_{ij}^2 u(x_2)
		-\lambda  \partial_{ij}^2 u(x_1) -(1-\lambda) \partial_{ij}^2u(x_3) \right)\leq 0.
			}
	Using the equation \eqref{eqnu} we obtain 
	\bgs{
			0 \geq&\;  b(x_2, u(x_2),z) - \lambda b(x_1, u(x_1),z)- (1-\lambda) b(x_3, u(x_3),z).
				}
	Therefore, we get that 
	\bgs{ 
		&b(x_2,u(x_2),z ) - b(x_2, \lambda u(x_1) +(1-\lambda)u(x_3),z) 
		\\
		 \leq &\; \lambda b(x_1,u(x_1),z) + 
		(1-\lambda) b(x_3,u(x_3),z) - b(x_2, \lambda u(x_1) +(1-\lambda)u(x_3),z)  
				\\
		= &\; 
		{\C_{-b(\cdot, u(\cdot), z)}( x_1,x_3, \lambda) }
		.
		}
		 \\
		Using the mean value theorem of Lagrange, we have that there exists $\xi$ between $u(x_2)$ and  $\lambda u(x_1)+(1-\lambda) u(x_3)$ such that
	 \[ 
	 	\partial_u b(x_2,\xi,z) \left( u(x_2)-\lambda u(x_1)-(1-\lambda) u(x_3)\right) \leq 
	 	{\C_{-b(\cdot, u(\cdot), z)}( x_1,x_3, \lambda) }
	 	,\]
	 	hence, since $\C_u$ is positive at $(x_1,x_3,\lambda),$ it follows that
	 	 \eqlab{\label{aaah}
		\mathcal C_u(x_1,x_3,\lambda) \leq \frac1\beta \, 
		{\C_{-b(\cdot, u(\cdot), z)}( x_1,x_3, \lambda) }.
	} 
	This concludes the proof of the lemma.
	\end{proof}

	Roughly speaking the previous statement says that under the assumption that the function
	$\C_u$ achieves a positive maximum in the {\em interior} of $\Omega\times \Omega \times [0,1]$, then this maximum is bounded from above by the convexity function of $-b$ {along $u$} (with $b$  strictly  increasing), computed at the interior maximum point of $\C_u$.

We recall now a result \cite[Theorem 2]{HU} for $\delta$-convex functions.

\begin{proposition}[Hyers-Ulam Theorem]
	\label{Ulam}
	Let $X$ be a space of finite dimension and $D\subset X$ convex. 
	Assume that $f:D\to\R$ is $\delta$-convex.
	Then there exists $g :D\to\R$ a convex function
	 such that $\|f-g\|_{L^\infty(D)}\leq \delta k_n$, where
	$k_n>0$ depends only on $n=\dim(X)$.
\end{proposition}
\vskip3pt
The following is the main $\delta$-convexity tool for applications.
It states that  the {\em approximate concavity} of $b$ {along $u$} and the strict monotonicity of $b$
	yields in turn the {\em approximate convexity} of the solution $u$.  
\begin{theorem}[$\delta$-Convexity Principle I] 	\label{first}
	Let $u\in C^2(\Omega)\cap C(\bar\Omega)$ be a solution of \eqref{eqnu} and set $M:=\|u\|_{C^2(\Omega)}$, $m=\|u\|_{L^\infty(\Omega)}$.
For some $\delta \geq 0$ and $\beta>0$ we assume that condition \eqref{symmatr} holds, and furthermore, that
		\eqlab{ \label{ah1}
			  & \partial_s b(x,s,z) \geq  \beta, \; \mbox{ for any  } (x,s ,z) \in \Omega \times[-m,m] \times \bar B_M,
			}
			  \eqlab{ \label{ah21}
			  & 
			\C_{-b(\cdot, u(\cdot), z)}( x_1,x_3, \lambda) 			  \leq \delta,
			  \; \mbox{ for any }  (x_1,x_3,\lambda) \in \Omega \times \Omega \times [0,1] \text{ and for all $z\in \bar{B}_M$}  . 
			}
	Then, if $\mathcal C_u$ achieves a positive interior maximum in $\Omega\times \Omega \times [0,1]$,  
	there exist a convex function $v\colon \Omega \to \R$ and $k_n>0$
		 such that
	\[ \|u-v\|_{L^\infty(\Omega)} \leq \frac{k_n}{ \beta} \delta
	.\]	
\end{theorem} 	
\begin{proof}
The proof is a consequence of Lemma \ref{mainlem} and Proposition \ref{Ulam}. 
\end{proof}

\begin{remark}
	For $\delta=0$ the assertion reduces exactly to the  Korevaar maximum principle (see \cite[Theorem 1.3, Lemma 1.4]{kore}). 
\end{remark}

\begin{remark}
One can obtain a statement similar to Theorem \ref{first} in the parabolic case (check \cite[Theorem 1.6]{kore}). Indeed,
consider the problem 
	\eqlab{ \label{parabeqnu} \partial_t u(t,x) =\mathcal L u(t,x) \quad \text{in $(0,T]\times \Omega$},\qquad  \mathcal L u(t,x):= \sum_{i,j=1}^n a^{ij}(t,Du) \partial^2_{ij} u(t,x) -b(t,x,u,D u).}
	Let $u$ be a solution of \eqref{parabeqnu} such that $u(t,\cdot) \in C^2(\Omega)\cap C(\bar\Omega)$ for any $t\in (0,T]$ and $u(\cdot, x)\in C((0,T])$. Assume that 
		\bgs{
						 A=[a^{ij}(t,\xi)]_{i,j} \mbox{ is symmetric, positive defined for all } \xi\in\R^n, \, \mbox{ for any fixed $t\in (0,T]$} }
				and denote
	\bgs{ 
	&{\C_u(t,x_1,x_3,\lambda):= \C_{u(t,\cdot)}(x_1,x_3,\lambda) = u(t, x_2)- \lambda u(t,x_1)-(1-\lambda)u(t,x_3)}
	}
	for any fixed $t$.
Then, if $\C_u$ achieves a positive maximum at $(t_0,x_1,x_3,\lambda )\in (0,T]\times \Omega\times \Omega \times [0,1]$ and there exists $\beta >0$ such that
	\eqlab{\label{mon}
	\inf_{\xi \in [u(t_0,x_2), \lambda u(t_0,x_1) +(1-\lambda)u(t_0, x_3)] } \partial_s b(t_0,x_2,\xi ,Du(x_1)) \geq  \beta 
	}
then 
\[\C_u(t_0,x_1,x_3,\lambda ) \leq \frac1{\beta}\, 
{\C_{-b(t_0,\cdot, u(t_0,\cdot), Du(x_1))}( x_1,x_3, \lambda) }	
.
\]
To see this, it is enough to substitute the equation \eqref{parabeqnu} into \eqref{kkk1}, obtaining that
\bgs{
			0 \geq&\;  b(t_0, x_2, u(t_0, x_2),z) - \lambda b(t_0, x_1, u(t_0, x_1),z)- (1-\lambda) b(t_0, x_3, u(t_0, x_3),z)
			\\
			&\; + u_t(t_0,x_2) -\lambda u_t(t_0, x_1)- (1-\lambda ) u(t_0, x_3).
				}
				We use the fact that $\C_u$ has a maximum in $t_0\in(0,T]$, getting
\[ \partial_t \C_u (t,x_1,x_3,\lambda)\, \vline_{\, t=t_0} =u_t(t_0,x_2)-\lambda u_t(t_0,x_1)-(1-\lambda) u_t(t_0,x_3)\geq 0,\] and from there the proof follows as in Lemma \ref{mainlem}. 
The analogue of Theorem \ref{first} is obtained by imposing that the function $b(t,x,s,z)$ be jointly $\delta$-convex {along $u$}  for any $z\in \bar B^t_M$, with $M=\|u(t,\cdot)\|_{C^2(\Omega)}$ and any fixed $t\in (0,T]$. In other words, there exist $\delta\geq 0, \beta >0$ such that
\bgs{
			  & \partial_s b(t,x,s,z) \geq  \beta \mbox{ for any  } (t,x,s ,z) \in (0,T]\times \Omega \times[-m,m] \times \bar B^t_M,
			}
			  \bgs{ 
			  &
			  { \sup_{ (x_1,x_3,\lambda)\in \Omega \times \Omega \times [0,1]} \C_{-b(t,\cdot, u(t,\cdot),z)} (x_1,x_3,\lambda) \leq \delta,
			  \quad \text{for all $z\in \bar{B}^t_M$, and any $ t\in (0,T]$}}
			  . 
			}
\end{remark}

In the next theorem, we encompass the case in which $\beta$ (from Theorem \ref{first}) may reach zero. The proof follows that of Korevaar in \cite[Lemma 1.5, Theorem 1.4]{kore}, we provide here a complete proof. Namely, we consider a perturbation of the problem in way that will allow us to apply Theorem \ref{first} to the solution of the perturbed problem on a smaller domain. Notice that we obtain a significant result if $b$ is jointly convex (i.e., one gets that $u$ is convex, as in Korevaar's result).  When $\delta >0$ however, we are only able to provide a rate of convergence of  the solution of the perturbed problem to a convex function, whereas the solution of the perturbed problem converges uniformly to the solution of the initial problem. The precise result goes as follows. 
\begin{proposition}
	\label{thm2}
Let $u\in C^2(\Omega)\cap C(\bar \Omega)$
be a solution of \eqref{eqnu}
and  assume that \eqref{symmatr} and \eqref{ah1} holds for $\beta=0 $. Let $\Omega'\Subset\Omega$ be smooth. 
If $\mathcal C_u$ achieves a positive interior maximum in $\Omega'\times \Omega' \times [0,1]$, 
		then for every $\eta >0$ there exist $\delta_0(\eta,\Omega')$ such that for any $0<\delta<\delta_0(\eta,\Omega')$
		there exists a function $v_\delta \colon \Omega' \to \R$, a convex function 
		$w_\delta \colon \Omega' \to \R$  and $k_n>0$  such that
		whenever \eqref{ah21} holds, then
	\bgs{
&	\|u-v_\delta \|_{L^\infty(\Omega')}\leq \eta, 
	\\
&\|w_\delta-v_\delta \|_{L^\infty(\Omega')}\leq 	k_n\sqrt{\delta}.
}

\end{proposition} 	

\begin{proof}
For $\eps\in(0,1/2)$ small, there exists $w$ and $M>0$ (indipendent on $\eps$) such that the function $v$, given as
\[ 
v_\eps=u+\eps w, \qquad \qquad \|w\|_{C^{2,\alpha}(\Omega')}\leq M, 
\] 
solves the perturbed problem  
		\syslab[]{\label{vvv} &\sum_{i,j=1}^n a^{ij}(Dv) \partial^2_{ij} v = b(x,v,Dv) +\eps v && \mbox{ in } \Omega'\\
				&v=u&& \mbox{ on } \partial\Omega'.
		}					 
				Indeed, let us take a Taylor expansion in $\eps$.
					 For $a^{ij}(Du +\eps Dw)$ we get 
			 \[ a^{ij}(Dv)= a^{ij}(Du)+\eps \sum_{k=1}^n \partial_{p_k} a^{ij}(Du)\partial_k w + \eps^2 G_1(Dw)
			  ,\]
			 while for $b(x,u+\eps w, Du+\eps Dw)$
			 \bgs{ &\;b(x,v,Dv) = b(x,u,Du) + \eps \left(\sum_{k=1}^n \partial_{p_k} b(x, u ,Du) \partial_k w +  \partial_u b(x,u,Du)w \right)
			  + \eps^2 G_2(w, Dw)
			 .
			 }
Summing up we get that
			 \bgs{ 
			 	 \sum_{i,j=1}^n a^{ij}(Du) \partial^2_{ij}u &
			 	 + \eps \Bigg[ \sum_{i,j=1}^n a^{ij}(Du) \partial^2_{ij} w 
			 	 + \sum_{k=1}^n \partial_k w \Big( \partial_{p_k} a^{ij}(Du)\partial^2_{ij} u 
			 	 -\partial_{p_k}b(x,u,Du)\Big) 
			 	 \\
			 	 &\;-   \partial_u b(x,u,Du)  w \Bigg]
			 	 = b(x,u,Du)  + \eps u + \eps^2 G(w,Dw, D^2w)
			 	.}
			 	 In the above computation, $G_1,G_2,G$ represent the rest of order two of the Taylor expansions. 	Just to be precise,   for some $\xi \in (0,\eps)$ we have
			 \bgs{ G(w,Dw,D^2w)= &\; w+ 	\sum_{l,k=1}^n \partial_{p_kp_l} a^{ij}(Du+\xi Dw) \partial_k w\partial_l w 
				 + 	 \Bigg[\partial^2_u b(x,u+\xi w,Du+\xi Dw) w^2
			\\
			&\; + \sum_{k=1}^n \bigg( 2\partial_{up_k} b(x,x,u+\xi w,Du+\xi Dw) w \partial_k w
			 \\
			 &\; + \sum_{l=1}^n \partial_{p_kp_l} b(x,u+\xi w,Du+\xi Dw) \partial_k w\partial_l w\bigg) \Bigg]
			 + \sum_{k=1}^n \partial_{p_k} a^{ij}(Du)\partial^2_{ij} w\partial_k  w
			 .}
			 	Knowing that $u$ satisfies the equation \eqref{eqnu}, and dividing by $\eps>0$ we get that
			 	\bgs{ \sum_{i,j=1}^n  a^{ij}(Du) \partial^2_{ij} w &+ \sum_{k=1}^n \partial_k w \left(  \partial_{p_k} a^{ij}(Du) \partial^2_{ij} u -\partial_{p_k}b(x,u,Du)  \right)
			 	 -\partial_u b(x,u,Du) w
			 	 \\ =&\; u  + \eps G(w,Dw, D^2w).
			 	}
			 	Then $w$ solves the problem 
			 		\sys[]{ \tilde{\mathcal{L}} w &= \sum_{i,j=1}^n a^{ij}(x) \partial^2_{ij} w + \sum_{k=1}^n c_k(x)\partial_k w + d(x) w = f, && \mbox{ in  } \Omega'
			 		\\
			 		w&=0, && \mbox{ on } \partial \Omega',
			 		}
			 	with 
			 	\bgs{
			 		c_k(x)= \partial_{p_k} a^{ij}(Du) \partial^2_{ij} u -\partial_{p_k}b(x,u,Du), \qquad 
			 		f(x) =  u  + \eps G(w,Dw, D^2w),
			 	}
			 	and
			 	\[ d(x)= -\partial_u b (x,u,Du) \leq 0.\]
			 	By iteration  we will consider $w^1=0$ and take 
			 	\syslab[\label{pb1}]{ &\tilde{\mathcal{L}} w^{k+1} =  u+ \eps G(w^k,Dw^k,D^2 w^k)&& \mbox{ in } \Omega'
			 	\\
			 	& w^{k+1}=0 && \mbox{ on } \partial \Omega'.}
Notice that considering a problem 			 	
			 		\sys[]{ &\tilde{\mathcal{L}} w = f(x) && \mbox{ in } \Omega'
			 	\\
			 	& w=0 && \mbox{ on } \partial \Omega',}
			 	by Schauder estimates there exists $K_1>0$ such that
			 	\[\|w\|_{C^{2,\alpha}(\Omega')} \leq K_1 \|f\|_{C^{0,\alpha}(\Omega')}.\] Also, since $G\in C^1$, one has for $v\in C^{2,\alpha}(\Omega')$ that if
			 	\[ \|v\|_{C^{2,\alpha}(\Omega')}\leq K_2 \qquad \qquad \implies \qquad \|G(v,Dv,D^2v)\|_{C^{0,\alpha}(\Omega')} \leq K_3,\]
			 	for some $K_2, K_3>0$. Using these two remarks for the problem \eqref{pb1}, there exists $\eps_0>0$ such that for any $\eps \leq \eps_0$
			 	 	\[ \|w^k\|_{C^{2,\alpha}(\Omega')} \leq K_4.\] 
			 	Consider now the problem for $\eta^{k+1}=w^{k+1}-w^k$, namely
			 		\sys[]{ &\tilde{\mathcal{L}} \eta^{k+1} =  \eps \tilde G(\eta^k,D\eta^k,D^2 \eta^k) && \mbox{ in } \Omega'
			 	\\
			 	& \eta^{k+1}=0 && \mbox{ on } \partial \Omega'.}
		 	To get $\tilde G$, by Lagrange theorem, we have
		 	\bgs{
		 		& G(w^{k}, Dw^{k}, D^2w^{k})- G(w^{k-1}, Dw^{k-1}, D^2w^{k-1}) 
		 		\\
		 		=&\;  D G (\xi_k) \cdot (\eta^{k},D\eta^{k},D^2\eta^{k} ):= \tilde G (\eta^k, D\eta^k, D^2\eta^k),
		 		}
		 		for some $\xi_k \in \R\times \Rn \times \R^{2n}$ laying on the segment that unites the two arguments of $G$.
		 			 	Therefore we obtain that
		 			 	\[ \|\eta^{k+1}\|_{C^{2,\alpha}(\Omega')} \leq \eps K_1 \|\tilde G\|_{C^{0,\alpha}(\Omega')} \leq \eps K \|G\|_{C^1} \|\eta^k\|_{C^{2,\alpha}(\Omega')},
		 			 	\]
		 			 	hence for $\rho<1$ (since $\eps$ is arbitrarily small)
		 	\[ 
		 		\|w^{k+1}-w^k\|_{C^{2,\alpha}(\Omega')} \leq \rho \|w^{k}-w^{k-1}\|_{C^{2,\alpha}(\Omega')} .
		 		\]
		 		Therefore there exists $w\in C^{2,\alpha}(\Omega')$ such that $w_k \to w$ in $C^{2,\alpha}(\Omega')$ with
		 		\[ \|w\|_{C^{2,\alpha}(\Omega')} \leq K_4.\]
We apply to $v_\eps$ as the solution of \eqref{vvv} Theorem \ref{first} (where $b(x,v,Dv)$ from Theorem \ref{first} is given by $b(x,v,Dv) + \eps v$ in our case). Then
\[ \partial_v b(x,v,Dv )+ \eps \geq \eps >0. 
\]
By Theorem \ref{first} there exists a convex function $w_\eps$ such that
$$
\|w_\eps-v_\eps\|_{L^\infty(\Omega')}\leq k_n\frac{\delta}{\eps} .
$$
Set $\eps=\sqrt{\delta}$. Then $\|w_\delta-v_\delta\|_{L^\infty(\Omega')}\leq k_n\sqrt{\delta} .$
Of course $u=\lim_{\delta\to 0} v_\delta$. Then the assertion follows.
		 	\end{proof}

In the next lemma, we relax the conditions we ask to the nonlinear term. Following the work in \cite{kenn,kaw}, we can ask the function $b$ to be $\delta$-harmonic concave and obtain anyways the $\delta$-convexity of the solution to the problem \eqref{eqnu}. As a matter of fact, we can estimate the convexity function  of the solution by the harmonic concavity function of the nonlinear term and its rate of monotonicity. 
\begin{lemma}\label{mainlem2}
	Let $u\in C^2(\Omega)\cap C(\bar\Omega)$ be a solution of \eqref{eqnu} 
	and assume that \eqref{symmatr} holds.	
Then, if $\C_{u}$ achieves a positive interior  maximum at $(x_1,x_3,\lambda )\in \Omega\times \Omega \times [0,1]$ and there exists $\beta >0$ such that \eqref{mon} holds, 
then 
\[\C_{u}(x_1,x_3,\lambda ) \leq -
{\frac {\HC_{b(\cdot,u(\cdot), Du(x_1))}(x_1,x_3, \lambda) } {\beta}}
.
\]
\end{lemma} 	
We follow in the next proof the main ideas from \cite[Theorem 3.1]{kenn} (another proof is given in \cite[Theorem 3.13]{kaw}).
	\begin{proof} If $x_1=x_3$, or $\lambda=0$, or $\lambda=1$ then $\C_u =0$ and there is nothing to prove. In the other cases
	as in Lemma \ref{mainlem}, we notice that we have that
	\[ Du(x_1)=Du(x_2)=Du(x_3)=:z\]
	and we name the matrix $A:= [a_{ij}(z)]$. Let us also define
	the $2n\times 2n$ matrices
	\bgs{
	 C:= [D^2 \C_u (x_1,x_3,\lambda)]  =
	 \begin{bmatrix}
	 	D^2_{x_1} \C_u(x_1,x_3,\lambda) &D^2_{x_1,x_3} \C_u (x_1,x_3,\lambda)  \\
	 	D^2_{x_1,x_3} \C_u (x_1,x_3,\lambda)  & D^2_{x_3} \C_u (x_1,x_3,\lambda)  
	 \end{bmatrix}
	 }
	  (which is negative defined since $(x_1,x_3,\lambda)$ is a maximum for $\C_u$ in the interior of its domain), and 
	\bgs{ B:=
	\begin{bmatrix} s^2 A & stA \\
					stA & t^2A
					\end{bmatrix}
					}
					for any $s,t\in \R$ (which is positive defined, since $A$ is so). We have from linear algebra arguments (see i.e. \cite[Lemma A.1]{kenn}) that $\Tr (BC)\leq 0.$ This means that
					\[ s^2 \Tr(A D^2_{x_1} \C_u ) + t^2 \Tr (A D^2_{x_3} \C_u ) + 2st \Tr (A D^2_{x_1,x_3} \C_u)\leq 0.\]
					Denoting
					\[\alpha=  \Tr(A D^2_{x_1} \C_u )  , \qquad  \gamma= \Tr (A D^2_{x_3} \C_u ),  \qquad \beta= \Tr (A D^2_{x_1,x_3}\C_u) \]
					it holds that
					\[ s^2 \alpha + t^2 \gamma +2 st \beta \leq 0,\] 
			thus 
					\eqlab{ \label{alpha} \alpha  \leq 0 , \qquad  \gamma \leq0 , \qquad \beta^2-\alpha \gamma \leq 0.
					}
	Using as in the proof of \cite[Theorem 3.1]{kenn} 
	\[Q_\eta= \sum_{i,j=1}^n a^{ij}(Du(\eta))\partial^2_{ij} u(\eta)\] 
we compute
	\bgs{
	& \alpha= \sum_{i,j}^n a^{ij}(z) \left(  \lambda ^2 \partial^2_{ij}u(x_2) - \lambda \partial^2_{ij}u(x_1) \right) = \lambda^2 Q_{x_2} - \lambda Q_{x_1}
	\\
	&\gamma=  \sum_{i,j}^n a^{ij}(z) \left(  (1-\lambda)^2 \partial^2_{ij}u(x_2) -(1- \lambda) \partial^2_{ij}u(x_3) \right) = (1- \lambda)^2 Q_{x_2} -(1- \lambda) Q_{x_3}
	\\
	&\beta  =\lambda(1-\lambda) \sum_{i,j}^n a^{ij}(z)  \partial^2_{ij}u(x_2)  = \lambda(1-\lambda)Q_{x_2}.
	}
		This together with \eqref{alpha} leads to 
		\[ Q_{x_2} ( \lambda Q_{x_3} +(1-\lambda)Q_{x_1} ) \leq Q_{x_1}Q_{x_3}, \qquad Q_{x_2}\leq \frac{1}{\lambda} Q_{x_1}, \qquad Q_{x_2} \leq \frac{1}{1-\lambda}Q_{x_3}.\]
		If $\lambda Q_{x_3} +(1-\lambda)Q_{x_1} \leq 0$ then
		\[ Q_{x_1}Q_{x_3} \geq Q_{x_2} ( \lambda Q_{x_3} +(1-\lambda)Q_{x_1} ) \geq  \frac{1}{\lambda} Q_{x_1}  ( \lambda Q_{x_3} +(1-\lambda)Q_{x_1} )  = Q_{x_1}Q_{x_3}+\frac{1-\lambda}{\lambda} Q^2_{x_1} \]
		hence $Q_{x_1}=0$, and in the same way $Q_{x_3}=0$. Then it can happen that either
		\eqlab{ \label{un} \lambda Q_{x_3} +(1-\lambda)Q_{x_1} \leq 0 \quad  \implies \quad Q_{x_1}=Q_{x_3}=0,\,\,\, Q_{x_2} \leq 0,}
		  or 
		 \eqlab{ \label{du}  \lambda Q_{x_3} +(1-\lambda)Q_{x_1} > 0 \quad  \implies \quad 
			Q_{x_2} \leq \frac{Q_{x_1} Q_{x_3}}{(1-\lambda)Q_{x_1} +\lambda Q_{x_3}},}
				(see also \cite[(3.5)]{kenn}, but we remark that the notations and signs there are different). By using the equation \eqref{eqnu} it holds that
	\[Q_{\eta}= b(\eta,u(\eta),Du(\eta)) \quad \mbox{ for } \eta=x_1,x_2,x_3,\]
	hence we get in the case  \eqref{du} 
	\bgs{
		b(x_2, u(x_2),z) \leq \frac{b(x_1,u(x_1),z) b(x_3,u(x_3),z)}
		{(1-\lambda) b(x_1,u(x_1),z) + \lambda b(x_3, u(x_3), z)}.
}
Then 
\bgs{
	&b(x_2, u(x_2),z) - b(x_2, \lambda u(x_1)+ (1-\lambda) u(x_3)
	  ,z) \\
	  \leq&\;  \frac{b(x_1,u(x_1),z) b(x_3,u(x_3),z)}
		{(1-\lambda) b(x_1,u(x_1),z) + \lambda b(x_3, u(x_3), z)} - b(x_2, \lambda u(x_1)+ (1-\lambda) u(x_3)
	  ,z)
	  \\
	   = &\; -
	   { \HC_{b(\cdot,u(\cdot),z)}(x_1,x_3, \lambda) }
	   .
	}
		By Lagrange's mean value theorem, there exists some $\xi \in [u(x_2),  \lambda u(x_1)+ (1-\lambda) u(x_3)]$ such that
	\[  \partial_u b(x_2, \xi, z) \C_u (x_1,x_3,\lambda) \leq -
	{ \HC_{b(\cdot,u(\cdot),z)}(x_1,x_3, \lambda) }
	.\]
	Notice also that in the case \eqref{un}, since $Q_{x_2} \leq 0$, one gets that
	\[  \partial_u b(x_2, \xi, z) \C_u (x_1,x_3,\lambda) \leq  - b(x_2, \lambda u(x_1)+ (1-\lambda) u(x_3)
	  ,z) =-
	  { \HC_{b(\cdot,u(\cdot),z)}(x_1,x_3, \lambda) }
	  .\] 
	Since $\C_u (x_1,x_3,\lambda)>0$, in any case it follows that
	\[ \beta \C_u (x_1,x_3,\lambda) \leq \partial_u b(x_2, \xi, z) \C_{u} (x_1,x_3,\lambda) \leq -
	{ \HC_{b(\cdot,u(\cdot),z)}(x_1,x_3, \lambda) }
	,\]
	therefore
	\[ \C_{u} (x_1,x_3,\lambda) \leq - \frac{
	{ \HC_{b(\cdot,u(\cdot),z)}(x_1,x_3, \lambda) }
	}\beta .\]
		This concludes the proof of the Lemma.
	\end{proof}
With the aid of this Lemma, we can obtain the second $\delta$-convexity principle, that we state in the next rows.
\begin{theorem}[$\delta$-Convexity Principle II] 
	\label{thm3}
	\label{second}	Let $u\in C^2(\Omega)\cap C(\bar\Omega)$ be a solution of \eqref{eqnu} and set $M:=\|u\|_{C^2(\Omega)}$, $m=\|u\|_{L^\infty(\Omega)}$.
For some $\delta \geq 0$ and $\beta>0$ we assume that condition \eqref{symmatr} holds, and furthermore, that
		\bgs{ \label{ah2}
			   \partial_s b(x,s,z)  \geq \, \beta,& \; \mbox{ for any  } (x,s ,z) \in \Omega \times[-m,m] \times \bar B_M,
			 }
						 \bgs{ \label{ah2}	  
			 			  {  \HC_{b(\cdot,u(\cdot),z)}(x_1,x_3, \lambda) } 
			  \geq &-\delta,  \; \text{ for all }  (x_1,x_3,\lambda)\in \Omega\times \Omega \times [0,1] 
			  \mbox{ for which one of }  
			  \\ & \quad \quad\mbox{conditions \eqref{hcon} holds, and all } z\in \bar B_M
			  .
			}
	Then, if $\mathcal C_{u}$ achieves a positive interior maximum in $\Omega\times \Omega \times [0,1]$,  
	there exist a convex function $v\colon \Omega \to \R$ and $k_n>0$
		 such that 
	\[ \|u-v\|_{L^\infty(\Omega)}\leq \frac{k_n}{ \beta} \delta
	.\]	
\end{theorem} 	
\begin{proof}
The proof is a consequence of Lemma \ref{mainlem2} and Proposition \ref{Ulam}. 
\end{proof}

	Theorem \ref{thm3} says that under the assumption that the function
	$\C_u$ achieves a positive interior maximum, then the {\em approximate  harmonic concavity} of $b$ and the strict monotonicity of $b$ ($b$ needs to be strictly  increasing) 
	yields in turn the {\em approximate convexity} of the solution $u$.

\section{Boundary constraints}\label{bdry}
In this section, we present some results that will allow us to exclude the possibility that the maximum of the convexity function of the  solution to \eqref{eqnu} is reached on the boundary.
Let us mention that a general framework for boundary constraints is given in \cite[Lemma 2.1]{kore}. We focus here on some particular cases, that will allow us to apply in a simple way our approximate convexity principles.
We recall that the definition of boundary point for the convexity function $\C_u$ is given in Definition \ref{x2}.

\begin{proposition}
	\label{boundary}
	Let $n\geq 2$, $\Omega$ a bounded convex domain of $\R^n$  and $u\in C(\bar\Omega)$ such that $u=0$ on $\partial\Omega$, $u>0$ in $\Omega$ and	
	 for every $y\in\partial\Omega$ and any $z\in \Omega$, there holds
	\begin{equation}
	\label{cond}
	\limsup_{t\to 0^+} t^{-1/\alpha}u(y+t(z-y))>u(z).
	\end{equation}
	Then for any $\alpha\in (0,1)$ the function $\mathcal C_{-u^\alpha}$
	cannot achieve the positive maximum on the boundary.
\end{proposition}
\begin{proof}
	Assume by contradiction that the positive maximum of the
	function $-\mathcal C_{u^\alpha}$ is achieved at a boundary  
	 point $(y,z,\lambda )$.
	
	\noindent Notice that if at least two of the points $y,z , \lambda y+(1-\lambda)z$ are on $\partial \Omega$ then 
	$\mathcal C_{-u^\alpha}\leq 0$,
	hence there is nothing to prove.  In view of the previous consideration, we can reduce to the case in which 
	$y\in\partial\Omega$ and $z\in \Omega$ (the fact that $y,z\in \Omega$ and $\lambda y+(1-\lambda) z \in \partial \Omega$ is excluded by the convexity of the domain).
	 
\noindent Now, the condition \eqref{cond} is equivalent to 
$$
\limsup_{t\to 0^+} \frac{u^\alpha(y+t(z-y))}{t}>u^\alpha(z).
$$

\noindent There exists $\tau\in (0,\lambda)$ sufficiently small 
that $u^\alpha(\tau z +(1-\tau)y)>\tau u^\alpha(z).$ Then, setting
$$
\xi:=\tau z +(1-\tau)y \in\Omega,\qquad \mu:=\frac{\lambda-\tau}{1-\tau}\in (0,1)
$$
we have $\mu z +(1-\mu) \xi =\lambda z +(1-\lambda)y$ and 
\begin{align*}
\mathcal C_{-u^\alpha}(z,\xi,\mu)&=\mu u^\alpha(z) +
(1-\mu) u^\alpha(\xi)-u^\alpha(\mu z+(1-\mu)\xi), \\
&=
\mu u^\alpha(z) +(1-\mu) u^\alpha(\xi)-u^\alpha(\lambda z+(1-\lambda)y), \\
&>(\mu+(1-\mu)\tau ) u^\alpha(z)-u^\alpha(\lambda z+(1-\lambda)y)
\\
&=\lambda u^{\alpha}(z) -u^\alpha(\lambda z+(1-\lambda)y)
 =\mathcal C_{-u^{\alpha}} (z,y,\lambda),
\end{align*}
which yields a contradiction.
\end{proof}
%

The next result will be very useful in applications. We will denote by $\frac{\partial u}{\partial n}$
the normal derivative where $n$ stands for the outer normal vector at the boundary.

\begin{corollary}
	\label{boundary3}
	Let $n\geq 2$, $\Omega$ a bounded convex domain of $\R^n$  and $u\in C^1(\bar\Omega)$ such that
	$$
	\text{$u>0$ on $\Omega$, \qquad $u=0$ on $\partial\Omega$, \qquad  $\frac{\partial u}{\partial n}<0$ on $\partial\Omega$}.
	$$
	Then for any $\alpha\in (0,1)$ the function $\mathcal C_{-u^\alpha}$
	 cannot achieve the positive 
	 maximum on  the boundary.
\end{corollary}
\begin{proof}
	In light of Proposition \ref{boundary}, it is enough to 
	prove that for any $y\in\partial\Omega$ and $z\in\Omega$, it holds
	\begin{equation*}
	\limsup_{t\to 0^+} t^{-1/\alpha}u(y+t(z-y))>u(z).
	\end{equation*}
	Since $\alpha\in (0,1)$, by convexity of $\Omega$ and $\frac{\partial u}{\partial n}<0$ we have
	\begin{align*}
	\limsup_{t\to 0^+} t^{-1/\alpha}u(y+t(z-y)) & =\limsup_{t\to 0^+} t^{(\alpha-1)/\alpha}
	\frac{u(y+t(z-y))-u(y)}{t} \\
	& =D u(y)\cdot (z-y)\lim_{t\to 0^+} t^{(\alpha-1)/\alpha}=+\infty,
	\end{align*}
	which yields the assertion.
	\end{proof}

Let us also mention that:
\begin{lemma}
	\label{boundary2}
	Let $n\geq 2$, $\Omega$ a bounded convex domain of $\R^n$  and $u\in C^1(\bar\Omega)$ such that $u=0$ on $\partial\Omega$, $u>0$ in $\Omega$ and
	\begin{equation}
	\label{cond2}
	\liminf_{t\to 0^+} t^{-1/\alpha}u(y+t(z-y))=0.
	\end{equation}
	for some $y\in\partial\Omega$ and $z\in\Omega$.
	Then there exists $\delta>0$ such that $u^\alpha$ is not $\delta$-concave.
\end{lemma}
\begin{proof}
	Assume by contradiction that for every $\delta>0$, $u^\alpha$ is $\delta$-concave. Then, since $u(y)=0$, we have
	$$
	u^\alpha(y+t(z-y))\geq tu^\alpha(z)-\delta.
	$$
	Letting $\delta=t\delta_0$ with $\delta_0\in (0,u^\alpha(z))$ and dividing by $t$ yields
	$$
	t^{-1/\alpha}u(y+t(z-y))\geq (u^\alpha(z)-\delta_0)^{1/\alpha},
	$$
	which gives a contradiction as $t$ goes to zero.
	\end{proof}



	  For the next lemma we refer the reader to 
	\cite[Lemma 3.12]{kaw}.
	\begin{lemma}\label{kawo}
	Let $\Omega \subset \Rn $ be a bounded and strictly convex domain (i.e., if $x_1\neq x_3\in \partial \Omega$ then $x_2\in \Omega$) with boundary of class $C^1.$ Let $u\in C^1(\bar\Omega)$ such that
	$$
	\text{$u>0$ on $ \Omega$, \qquad $u=0$ on $\partial\Omega$, \qquad  $\frac{\partial u}{\partial n}<0$ on $\partial\Omega$}.
	$$Let $f\colon \R^+\to \R$ be a $C^1$ functions that satisfies
	\begin{equation} 
	\label{Kawcond}
		f' <0 \qquad \mbox{ and } \lim_{u\searrow 0} f'(u)=-\infty.
		\end{equation}
	Then $\C_{f(u)}$ cannot achieve a positive maximum on the boundary.
 	\end{lemma}
	
	For instance $f(s)=-\log s$ satisfies conditions \eqref{Kawcond}.
	
\section{$\delta$-concave solutions}\label{examples}
In this section, we give some applications of the $\delta$-convexity principles established in Theorems \ref{first} and \ref{thm3}.

The next results is a meaningful application of our general
	 results. It contains in particular semi-linear eigenvalue problems.

\begin{theorem}[$f$-convex solutions]
	\label{fconv}
	Let $n\geq 2$,  $f\in C^2(\R^+)$ be such that it satisfies \eqref{Kawcond} and
	in addition, that
	$$
	\text{the function
	$s\mapsto \frac{f''(f^{-1}(s))}{[f'(f^{-1}(s))]^2}$ is increasing and concave.}
	$$ 
	Let $\Omega$ be a $C^1$ bounded strictly convex 
	domain of $\R^n$ and $u\in C^2(\Omega)\cap C(\bar\Omega)$ be a solution to $u=0$ and 
	$\partial u/\partial n<0$ on $\partial\Omega,$ $u>0$ in $\Omega$ and
	$$
	\sum_{i,j=1}^{n}a_{ij}(-f'(u)Du)\partial^2_{ij}u=\frac{1}{f'(u)}b(x,f(u),-f'(u)Du) \qquad \text{in $\Omega$}.
	$$
	We suppose furthermore that
	denoting $M=\|u\|_{C^2(\Omega)}$, $m=\|u\|_{L^\infty(\Omega)}$, there exists $\beta>0, \delta \geq 0$ such that
	\bgs{ \label{ah1-ext}
			& \partial_s b(x,s,z)
		\geq  \beta \mbox{ for any  } (x,s ,z) \in \Omega \times[f(m),\lim_{x\searrow 0} f(x)) \times \Rn ,
		\\
		&
		{{ \sup_{ (x_1,x_3,\lambda)\in \Omega\times\Omega 
		 \times [0,1]} \C_{-b(\cdot, f(u(\cdot)),z)}
			(x_1,x_3,\lambda) }}
		\leq \delta, \quad \text{for all } z\in \Rn
		.
	}  
	Then there exists a 
	convex function $v\colon \Omega \to \R$  and $k_n>0$ such that
	\[	\|f(u)-v\|_{L^\infty(\Omega)} \leq \frac{k_n\delta}\beta .\]
\end{theorem}
\begin{proof}
Setting $w=f(u)$, a standard computation shows that $w$ satisfies the problem
	$$
\sum_{i,j=1}^{n}a_{ij}(-Dw)\partial_{ij}^2w=b(x,w,-Dw)+\frac{f''(f^{-1}(w))}{[f'(f^{-1}(w))]^2}\sum_{i,j=1}^{n}a_{ij}(-Dw)\partial_{i}w\partial_jw.
$$
Notice that, by assumption on $f$, we have that the function 
$$
b_1(x,w,-Dw) := b(x,w,-Dw)+\frac{f''(f^{-1}(w))}{[f'(f^{-1}(w))]^2}\sum_{i,j=1}^{n}a_{ij}(-Dw)\partial_{i}w\partial_jw
$$
is monotonically increasing in $w$, and its derivative is greater or equal than $\beta$. The  function 
$$
b_2(w, -Dw) :=\frac{f''(f^{-1}(w))}{[f'(f^{-1}(w))]^2}\sum_{i,j=1}^{n}a_{ij}(-Dw)\partial_{i}w\partial_jw
$$
is concave in $w$, thus $\C_{-b_2} \leq 0$. This yields that 
	\bgs{{{ \sup_{ (x_1,x_3,\lambda)\in \Omega\times\Omega 
		 \times [0,1]} \C_{-b_1(\cdot, w(\cdot), Dw)}
			(x_1,x_3,\lambda) }} 
			\leq &\;  \sup_{ (x_1,x_3,\lambda)\in \Omega\times\Omega 
		 \times [0,1]} \C_{-b_2( w(\cdot), Dw)} (x_1,x_3,\lambda)  
		 \\
		 &\; + \sup_{ (x_1,x_3,\lambda)\in \Omega\times\Omega 
		 \times [0,1]} \C_{-b(\cdot, w(\cdot), Dw)}  (x_1,x_3,\lambda) 
		 \\
		 \leq&\; \delta
		}
		 for any $Dw$, by hypothesis. 
Notice also that 
according to Lemma \ref{kawo} the convexity function $\mathcal C_w$ cannot achieve a positive maximum on the boundary. 
Thus the maximum is reached in the interior of the domain. It follows by  Theorem \ref{first}
that there exist a convex function $v\colon \Omega \to \R$ and $k_n>0$
such that
	\[
	\|f(u)-v\|_{L^\infty(\Omega)} \leq \frac{k_n}{ \beta}\delta.\]
	This concludes the proof of the Theorem. 
	\end{proof}

\begin{corollary}\label{logconc}
Let $\Omega$ be a $C^1$ bounded strictly convex 
	domain of $\R^n$ and $u\in C^2(\Omega)\cap C(\bar\Omega)$ be a solution to $u=0$ on $\partial\Omega,$ $u>0$ in $\Omega$ and
	$$
	\Delta u+\lambda u- ug(u)=0 \qquad \text{in $\Omega$}.
	$$
	Let $\lambda,\delta,c>0$, $m=\|u\|_{L^\infty(\Omega)}$ and $g\in C^1((0,m],\R^+)$ with 
	\bgs{	 \qquad \qquad g\leq \lambda , \qquad g'(t)t\geq c, }
	and
	\bgs{ 
	 \C_{h( u(\cdot))}
			(x_1,x_3,\lambda)  \leq  \delta, \quad with \;\; h(s)=g(e^{-s}), \; \mbox{ for any } (x_1,x_3,\lambda) \in \Omega\times\Omega\times[ 0,1].
	}
	Then there exists a 
	concave function $v\colon \Omega \to \R$  and $C:=C(n,c,\|u\|_{L^\infty(\Omega)} )>0$ such that
	\[	\|\log u-v\|_{L^\infty(\Omega)} \leq C\delta.\]
	\end{corollary}
\begin{proof}
By Hopf's Lemma, we get first of all that $\partial u/\partial n <0$ on $\partial \Omega$. With the choice 
	$$
	f(s)=-\log s,\qquad a_{ij}=\delta_{ij},\qquad b(x,s)=\lambda - g(e^{-s})
	$$
we find ourselves with the problem in Theorem \ref{fconv}. 
		We have that $b\colon \Omega \times[-\log m,+\infty)\to \R$,
	\[ \partial_s b(x,s) = g'(e^{-s}) e^{-s} \geq c \]
	and
		that for any $(x_1,x_3,\lambda) $
	\[\C_{-b}(x_1,x_3,\lambda) \leq \C_{h} (x_1,x_3,\lambda) \leq \delta.
	\]
	The assertion follows by Theorem \ref{fconv}.
	\end{proof}

In this  example, we take as the nonlinearity a perturbation of a concave function and prove that an appropriate power of the solution is  approximately concave, hence it can be written as a bounded perturbation of a concave function.

\begin{theorem}[Power concave solutions]
	\label{powerC}
	Let $n\geq 2$, $\gamma\in [0,1)$, $\Omega$ a bounded convex domain of $\R^n$ that satisfies an interior ball condition. Let $u\in C^2(\Omega)\cap C(\bar\Omega)$ be a solution to
	$$
	\Delta u+u^\gamma
	-	 u^{\frac{1+\gamma}2}g(u)=0,  \qquad \text{$u=0$ on $\partial\Omega$},
	\qquad \text{$u>0$ in $\Omega$.}
	$$
	Denoting $\|u\|_{L^\infty(\Omega)}=m$, we take here $g\in C^1((0,m], \R^+)$ is such that it holds that
	\eqlab{\label{mmmf}
	 g(s) \leq s^{\frac{\gamma-1}{2}}, \qquad g'(s)\geq 0
	 }
	and for some $\delta\geq 0$ 
	\[ \HC_h (s_1,s_3,\lambda) \leq \delta, \qquad \mbox { for any } s_1,s_3\in(0,m], \qquad \mbox{ with } \; h(s)= g(s^{\frac{2}{1-\gamma}}) .\] 
Then there exists a concave function $v$ and a positive constant $C:=C(n, m,\gamma)$ such that
\[
	\|u^{(1-\gamma)/2}-v\|_{L^\infty(\Omega)}\leq C\delta.
	\] 
\end{theorem}
 

\begin{proof}
	Notice that by \eqref{mmmf}
	\[
		u^\gamma \geq u^{\frac{1+\gamma}2}g(u), 
	 \]
	  so applying Hopf's Lemma, we deduce that 
	${\partial u}/{\partial n}<0$ on $\partial\Omega$. Consider now the transformation of $u$ given by 
	\[	w:=-u^{(1-\gamma)/2}.\]
	 By applying Corollary \ref{boundary3} with $\alpha:=(1-\gamma)/2\in (0,1)$ we have that  the convexity function $\mathcal C_w$
	cannot achieve the maximum on the boundary. Thus, the maximum is achieved in the interior of the domain. If such a maximum is non-positive,
	there is nothing to prove, since this yields that $w$ is convex. So we assume that the maximum is positive. 
	Observe that standard computations yield that $w$ satisfies 
	$$
	\Delta w-\tilde b(w,Dw)+\tilde g(w) =0,
	$$
	where   we have set $\tilde b: [-m^{\frac{1-\gamma}{2}},0)\times \R^n\to\R$, $\tilde g\colon [-m^{\frac{1-\gamma}{2}},0)\to \R$, with 
		$$
	\tilde b(s,z):=\left(\frac{1+\gamma}{1-\gamma}|z|^2+
	\frac{1-\gamma}{2}\right)\frac{1}{(-s)},\qquad 
	\tilde g(s):= \frac{1-\gamma}{2} 
	g\big((-s)^{\frac{2}{1-\gamma}}\big).
	$$
	Thanks to \eqref{mmmf} we have that
	\[ \partial_s (\tilde b(s,z) - \tilde g(s) )\geq \frac{1-\gamma}{2m^{1-\gamma}} >0 \qquad \mbox{ and that } \;
		\tilde b(s,z) \geq \frac{1-\gamma}2 \frac{1}{(-s)} \geq \tilde g (s).
	\]
	In view of \eqref{subadd1} (remark that the harmonic concavity function is well defined, since all functions involved are non-negative) it follows that for any $(s_1,s_3,\lambda)$ 
	 \[
	 	\HC_{\tilde b-\tilde g} (s_1,s_3,\lambda) \geq \HC_{\tilde b}(s_1,s_3,\lambda)  -\HC_{\tilde g}(s_1,s_3,\lambda)  .
	 	\]
	Given that $\tilde b>\frac{1-\gamma}{2m^{(1-\gamma)/2}} $ and that the map
	\[ s \mapsto \frac{1}{\tilde b(s,z)} 
	\]  is convex, it follows that 
	$\tilde b$ is harmonic concave, thus $\HC_{\tilde b}\geq 0$.  Therefore	
	 	 \[
	 	\HC_{\tilde b-\tilde g} (s_1,s_3,\lambda) \geq -\HC_{\tilde g}  \geq -\frac{1-\gamma}2  \delta,
	 	\]
	and by Theorem \ref{thm3}, $w$ is $m^{1-\gamma}\delta$ convex. The conclusion immediately follows.
\end{proof}

\begin{theorem}\label{fffh}
	Let $n\geq 2$,  $\Omega$ be a bounded convex domain of $\R^n$, that satisfies the interior ball condition
	. Let $u\in C(\bar \Omega)\cap C^2( \Omega)$ be a solution to
	\eqlab{ \label{fffh1}
	\Delta u+   f(x) - u^{\frac{1+\gamma}{1+2\gamma} } g(x)=0,  \qquad \text{$u=0$ on $\partial\Omega$},
	\qquad \text{$u>0$ in $\Omega$}.
	} 
	Here, $f,g\in C(\bar \Omega, \R^+)$ are such that there exists $\gamma\geq 1$,  $c>0$ and $\delta\geq0$   such that 
	\[	 f^\gamma \; \mbox{ is concave}, \qquad  f(x)\geq c\quad
	\mbox{ in $\Omega$,}
	\]   denoting $m=\|u\|_{L^\infty(\Omega)}$	
	\[
		f(x)\geq m^{\frac{1+\gamma}{1+2\gamma}} g(x) \quad \mbox{ in } \Omega
		\]
	and 
		\[ 
				\HC_g(x_1,x_3,\lambda) \leq \delta \quad \mbox{ for any } (x_1,x_3,\lambda)  \: \mbox{ in the interior of } \Omega. 
			\]	
	Then there exists a concave function and $C:=C(n, c,m,\gamma)$ such that  
	$v\colon \Omega \to \R$  
	\[ 
	\|u^{\frac{\gamma}{1+2\gamma}} -v\|_{L^\infty(\Omega)} \leq C \delta.
	\] 
\end{theorem}

\begin{proof} Let 
\[ w=-u^{\frac{\gamma}{1+2\gamma}}.\]
By Hopf's Lemma (notice that $f(x)- u^{\frac{1+\gamma}{1+2\gamma}} g(x)\geq 0 $ by hypothesis), we have that 
$\frac{\partial u}{\partial n}<0$ on $\partial \Omega$,
hence by Corollary \ref{boundary3}, $\C_w$ cannot achieve the maximum at a boundary point. It follows that the maximum of $C_w$ is achieved at an interior point. 
 	The function $w$ satisfies  the equation
	\[
		\Delta w-b(x,w,Dw)  =0
		\]
		 with $b \colon  \Omega \times  [-m^{\frac{\gamma}{1+2\gamma}},0)\times \R^n \to \R$,
			\[
	b(x,s,z):=\frac{(1+\gamma)|z|^2}{\gamma} (-s)^{-1}
		+\frac{\gamma}{1+2\gamma} (-s)^{- \frac{1}{\gamma}-1}f(x) - \frac{\gamma}{1+2\gamma} g(x)
		:=\tilde b_z(x,s) - \frac{\gamma}{1+2\gamma} g(x) .
	\]
	We have that 	
		 	\[
				\partial_s b(x,s,z)=  \frac{(1+\gamma)|z|^2}{\gamma}(-s)^{-2}  +\frac{1+\gamma}{1+2\gamma} f(x) (-s)^{-\frac1\gamma -2}  \geq  \frac{c(1+\gamma)}{m(1+2\gamma)} .
				\]
			We claim that $\tilde b_z$ is harmonic concave in the two variables $(x,s)$. Indeed, denoting
			\[
				\tilde b_z(x,s)  = (-s)^{-2}  \left({ \frac{(1+\gamma)|z|^2}{\gamma} (-s)}
		+\frac{\gamma}{1+2\gamma} (-s)^{- \frac{1}{\gamma}+1}f(x) \right)
		=:(-s)^{-2} h_z(x,s)
		\]
					we follow the next line of thought. Since $f^\gamma$ and  
			$\big( (-s)^{-\frac1\gamma +1}\big)^{\frac{\gamma}{1-\gamma}} $  are concave, from \cite[Property 8]{kenn} we have that  $f(x) (-s)^{-\frac1\gamma +1}$ is concave (basically, \cite[Property 8]{kenn}  says that if $f^\alpha$ and $g^\beta$ are positive concave functions, that $fg$ is $1/\alpha+1/\beta$ concave). Thus $h_z(x,s)$ is concave, as sum of two concave functions. Then using Proposition \ref{rapp}, we have that $\tilde b_z(x,s)^{-1}$ is convex. Employing Proposition \ref{ccc}, we get the claim that $\tilde b_z$ is harmonic concave. Thus $\HC_{\tilde b} \geq 0$. We use \eqref{subadd1} to obtain that
			\[ \HC_{b}\geq \HC_{\tilde b} -\frac{\gamma}{1+2\gamma} \HC_g \geq -\frac{\gamma\delta }{1+2\gamma} .
			\] 
It follows from Theorem \ref{thm3}  that $w$ is $\frac{\delta \gamma m}{c(1+\gamma)} $ convex, thus $u^{\frac{\gamma}{1+2\gamma} } $ is $\frac{\delta \gamma m}{c(1+\gamma)} $-concave.
This concludes the proof of the Theorem.
\end{proof}

\begin{remark}[Quasilinear equations]
Assume that $a:\R\to\R$ is a function of class $C^1$ such that there exists $\nu>0$ with $a(s)\geq \nu$
for all $s\in\R$. Let $\varphi:\R\to\R$ be the unique solution to
\begin{equation*}
\varphi'=\frac{1}{\sqrt{a\circ \varphi}},\qquad \varphi(0)=0,
\end{equation*}
which is smooth and strictly increasing. Consider the quasilinear problem
\begin{equation}
\label{prob}
\begin{cases}
\,\dvg(a(u)Du)-\frac{a'(u)}{2}|Du|^2+u^\gamma-	 u^{\frac{1+\gamma}2}g(u)=0   & \text{in $\Omega$,} \\
\noalign{\vskip2pt}
\,\text{$u=0$\quad on $\partial\Omega$,} \qquad  \text{$u>0$\quad in $\Omega$,}   &
\end{cases}
\end{equation}
Then, it is possible to associate to \eqref{prob} the semilinear problem
\begin{equation}
\label{prob-semi}
\begin{cases}
\Delta v+v^\gamma- v^{\frac{1+\gamma}{2}}h(v)=0  \quad\text{in $\Omega$,}    &  \\
\noalign{\vskip2pt}
\,\text{$v=0$\quad on $\partial\Omega$,}   \quad  \text{$v>0$\quad in $\Omega$,}  &
\end{cases}
\end{equation}
where 
$$
h(s)=\frac{\eta(s)-s^\gamma}{ s^{(1+\gamma)/2}},\qquad 
\eta(s)=\frac{\varphi(s)^\gamma-
 \varphi(s)^{\frac{1+\gamma}2}g(\varphi(s))}{\sqrt{a(\varphi(s))}}.
$$
In fact, a direct computation shows that if $v\in C^2(\Omega)$ is a classical solution to problem~\eqref{prob-semi},
then $u=\varphi(v)$ is a classical solution to problem~\eqref{prob} and vice versa.
In particular, one can apply Theorem~\ref{powerC} and get information about the approximate concavity of $\varphi^{-1}(u)^{(1-\gamma)/2}$
from the harmonic  concavity of $h$. Of course the concavity of the solution depends also upon $a$. 
\end{remark}

\bigskip
\appendix
\section{}

In this section, we give some properties related to $\delta$-harmonic concavity. In the first lemma, we establish a sub-additivity property of the harmonic concavity function.

\begin{lemma}\label{lemmaf}
	Let $f, g\colon \Omega\to \R$.  Then at all points $( (y_1,s_1),(y_3,s_3),\lambda)$ for which one of the conditions \eqref{hcon} holds for $f,g, f+g$
	\eqlab{\label{subadd} \HC_{f+ g}( (y_1,s_1),(y_3,s_3),\lambda)\leq \HC_f( (y_1,s_1),(y_3,s_3),\lambda)+ \HC_g( (y_1,s_1),(y_3,s_3),\lambda).}
	Furthermore, at all points $( (y_1,s_1),(y_3,s_3),\lambda)$ for which one of the conditions \eqref{hcon} holds for $f,g, f-g$, then 
	\eqlab{\label{subadd1} 
		\HC_{f- g}( (y_1,s_1),(y_3,s_3),\lambda)\geq \HC_f( (y_1,s_1),(y_3,s_3),\lambda)- \HC_g( (y_1,s_1),(y_3,s_3),\lambda).	
		}
\end{lemma}
	\begin{proof} Recalling that $y_2=\lambda y_1+ (1-\lambda)y_3$ (and $s_2$ is th convex combination of $s_1,s_3$), for simplicity, we write
\[ g_i=g(y_i,s_i), \quad f_i=f(y_i,s_i) \qquad \mbox{ for } i=1,2,3.\]
When for $y_1,y_3,s_1,s_3,\lambda $  we have $g_1=g_3=0$ or $f_1=f_3=0$ or all $g_1=g_3=f_1=f_3=0$ or $f_1=-g_1\neq 0,f_3=-g_3\neq 0$ then
\[
	\HC_{f+ g}( (y_1,s_1),(y_3,s_3),\lambda)-\HC_{f}( (y_1,s_1),(y_3,s_3),\lambda)  
		 = \HC_g ( (y_1,s_1),(y_3,s_3),\lambda).
\]
Otherwise, for $\lambda g_3+ (1-\lambda)g_1 >0$ and $\lambda f_3+ (1-\lambda)f_1>0$ we compute 
\bgs{& \HC_{f+ g}( (y_1,s_1),(y_3,s_3),\lambda)-\HC_{f}( (y_1,s_1),(y_3,s_3),\lambda)  
\\
=&\;  g_2-  \frac{(f_1+ g_1) (f_3+ g_3) }{ \lambda (f_3+ g_3) +(1-\lambda) (f_1+ g_1)} +  \frac{f_1 f_1 }{ \lambda f_3 +(1-\lambda) f_1}  
=
g_2 -  \psi(1) +\psi(0),
}
considering the function
\[
\psi(\delta):= \frac{(f_1+\delta g_1) (f_3+\delta g_3) }{ \lambda (f_{3} +\delta g_3) +(1-\lambda)  (f_1+\delta g_{1} )} .
\]
Denoting $h_i=(f+\xi g)(x_i,s_i)$ for $i=1,2,3$
we have that
	\[
		\psi'(\xi) =\frac{\lambda g_1 h_3^2 + (1-\lambda) g_3 h_1^2}{ \big( \lambda h_3+(1-\lambda)h_1\big)^2}.
	\]
We apply the Lagrange mean value theorem: for $\xi \in (0,1)$,
	 \[\psi(1)=\psi(0) + \psi'(\xi) \]
so we obtain
\bgs{
	  \frac{f_1 f_3 }{ \lambda f_3+(1-\lambda) f_1} 			-  \frac{(f_1+g_1)(f_3 +g_3) }{ \lambda (f_3+g_3)+(1-\lambda) (f_1+g_1)}  =- \frac{\lambda g_1 h_3^2 + (1-\lambda) g_3 h_1^2}{ \big( \lambda h_3+(1-\lambda)h_1\big)^2}.	
}
Hence, we get
\eqlab{ \label{stt}
&\HC_{f+ g}( (y_1,s_1),(y_3,s_3),\lambda)-\HC_{f}( (y_1,s_1),(y_3,s_3),\lambda)  
\\
		=&\;  g_2 - \frac{\lambda g_1 h_3^2 
		+ (1-\lambda) g_3 h_1^2}{ \big( \lambda h_3+(1-\lambda)h_1\big)^2} 
\\
	=&\; \HC_g( (y_1,s_1),(y_3,s_3),\lambda) 
		 +\frac{g_1g_3}{\lambda g_3 +(1-\lambda)g_1} - \frac{\lambda g_1 h_3^2 
		+ (1-\lambda) g_3 h_1^2}{ \big( \lambda h_3+(1-\lambda)h_1\big)^2} 
		\\
		=&\; \HC_g( (y_1,s_1),(y_3,s_3),\lambda)    - \lambda(1-\lambda) \frac{\big( g_3h_1-g_1h_3\big)^2}{\big( \lambda g_3 +(1-\lambda)g_1 \big) \big( \lambda h_3+(1-\lambda)h_1\big)^2}
		\\
		\leq&\;  \HC_g( (y_1,s_1),(y_3,s_3),\lambda),  
		}
which concludes the proof of \eqref{subadd}.

To prove \eqref{subadd1}, we use \eqref{subadd} for $f-g$ and $g$ and we obtain that
\[ 
	\HC_{f}=\HC_{f-g+g} \leq \HC_{f-g} +\HC_g,\]
	hence the result.
		This concludes the proof of the Lemma.
\end{proof}

As an outcome of the previous lemma, we obtain also the following estimates.
\begin{corollary}
Let $f,g\colon \Omega \to \R$.  If $g\in L^\infty(\Omega)$ then  at all points $( (y_1,s_1),(y_3,s_3),\lambda)$ for which one of the conditions \eqref{hcon} holds for $f, f+g$
	\[ \HC_{f+g}( (y_1,s_1),(y_3,s_3),\lambda) -\HC_f ( (y_1,s_1),(y_3,s_3),\lambda) \leq \|g\|_{\infty}.\]
	If furthermore there exist $\alpha, C,M>0$ such that  $\alpha \leq f\leq C$, $0\leq g\leq M$, then
	\[
		 \HC_{f+g} ( (y_1,s_1),(y_3,s_3),\lambda)-\HC_f ( (y_1,s_1),(y_3,s_3),\lambda)\geq -\frac{MC}{\alpha^2}.
	\]
\end{corollary}
\begin{proof}
The proof follows immediately by estimating $ {\lambda g_1 h_3^2 
		+ (1-\lambda) g_3 h_1^2} \big( \lambda h_3+(1-\lambda)h_1\big)^{-2}  $ in line two of formula \eqref{stt}.
\end{proof}

The next proposition is the approximate concavity adaptation  of \cite[Lemma A.2]{kenn}. 
\begin{proposition} \label{rapp} $\quad$
\begin{enumerate}	
				\item Let $c,C, m,\delta>0$ be constants and let $g\colon \Omega\times [-m,m]\to \R^+$ be $\delta$-concave and $ 2\delta \leq c<g\leq C$.
			Then the map $(s,x)\in \R\times \Omega\mapsto s^2 g(x,s)^{-1}$ is $\delta$-convex jointly in the two variables $(x,s)$.
			\item Let $g\colon \Omega\times [-m,m]\to \R^+$ be concave.
			Then the map $(s,x)\in \R\times \Omega\mapsto s^2 g(x,s)^{-1}$ is convex jointly in the two variables $(x,s)$.
			\end{enumerate}
			\end{proposition}
			\begin{proof} We take any  $(x_1,s_1), (x_3,s_3)$ and denote as usual $x_2=\lambda x_1 + (1-\lambda )x_3 $ and $s_2=\lambda s_1 + (1-\lambda )s_3 $ and $ g_i=g(y_i,s_i) \mbox{ for } i=1,2,3$. Then, given that  
			\[g_2 \geq \lambda g_1+(1-\lambda)g_3 -\delta\]
			(notice  that  the right hand side term is strictly positive) we obtain 
			\bgs{
			\frac{s_2^2 }{ g_2 } -\lambda 		\frac{s_1^2 }{ g_1} -(1-\lambda) \frac{s_3^2 }{ g_3} 
					\leq &\; \frac{s_2^2}{\lambda g_1  +(1-\lambda)g(x_3,s_3) -\delta}  -\lambda 		\frac{s_1^2 }{ g_1 } -(1-\lambda) \frac{s_3^2 }{ g_3} 
			\\
			= &\; \frac{-\lambda(1-\lambda) \Big(s_1g_3-s_2g_1\Big)^2  + \delta \Big(\lambda s_1^2g_3 +(1-\lambda) s_3^2g_1 \Big)}{(\lambda g_1  +(1-\lambda)g_3-\delta) g_1 g_3}
			\\
			\leq &\; \delta \frac{\lambda s_1^2g_3 +(1-\lambda) s_3^2g_1 }{(\lambda g_3  +(1-\lambda)g_3 -\delta) g_1 g_3 }.
			}
			We have that
			\[
				\lambda s_1^2g_3+(1-\lambda) s_3^2g_1 \leq m^2 C
			\qquad
			\mbox{ and }\qquad 
				(\lambda g_1+(1-\lambda)g_3 -\delta) g_1 g_3 \geq \frac{c^3}2 .
		\]
		It follows that
		\bgs{
			&\frac{s_2^2 }{ g_2} -\lambda 		\frac{s_1^2 }{ g_1} -(1-\lambda) \frac{s_3^2 }{ g_3}  \leq C_1 \delta,
			}	
		for $C_1=2m^2 C c^{-3}$,
			hence the conclusion.
						The second point is obvious  if one takes $\delta=0$.	
					\end{proof}

					It is a known result that if for a positive function $g$ we have that $g^{-1}$ is convex, then $g$ itself results harmonic concave.  We can establish an approximate concavity analogue if we take $g$ bounded from above.
					\begin{proposition}
						\label{ccc} $\quad$
				\begin{enumerate}	
				\item  Let $C, \delta>0$ be constants. Let $g\colon \Omega\times \R \to \R^+$ be $\delta$-concave and $0< g<C$. Then if $g^{-1}$ is jointly $\delta$-convex, then $g$ is $\delta$-harmonic concave. 
				\item  Let $g\colon \Omega\times \R \to \R^+$ be concave. Then if $g^{-1}$ is jointly convex, then $g$ is harmonic concave. 
				\end{enumerate}	
					 \end{proposition}
				 \begin{proof}
					Consider any  $(x_1,s_1), (x_3,s_3)$ and take $x_2=\lambda x_1 + (1-\lambda )x_3 $ and $s_2=\lambda s_1 + (1-\lambda )s_3 $, as usual and  $ g_i=g(y_i,s_i) \mbox{ for } i=1,2,3$. 
					Then putting
								\[ p:= \frac{\lambda }{g_1 } +\frac{1-\lambda}{ g_3},\qquad \qquad p\geq \frac{1}C ,\]
			we have by definition
			\[\C_{\frac1g} ((x_1,s_1),(x_3,s_3),\lambda)= \frac{1}{g_2} -p\qquad  \mbox{ 	and } \qquad
			 \HC_{g}((x_1,s_1),(x_3,s_3),\lambda) = g_2-\frac1p.\]
			Then 
			\[ \C_{\frac1g} ((x_1,s_1),(x_3,s_3),\lambda) \leq \delta \]
			implies that
			\[  \HC_g ((x_1,s_1),(x_3,s_3),\lambda) \geq - \frac{\delta}{p(\delta +p)} \geq -C^2{\delta}.
			\] 
			This concludes the proof of the proposition, as the second point corresponds to $\delta=0$ and it is easily seen.
\end{proof}

\bigskip

\end{document}